\RequirePackage{fix-cm}
\documentclass{svjour3}                     
\smartqed  
\usepackage{graphicx}%
\usepackage{multirow}%
\usepackage{amsmath,amssymb,amsfonts}%
\usepackage{bm}
\usepackage[title]{appendix}%
\usepackage{xcolor}%
\usepackage[margin=1in]{geometry}
\usepackage{setspace}
\usepackage{url,caption,subcaption}
\usepackage{hyperref}

\begin{document}

\title{A Family of Convex Models to Achieve Fairness through Dispersion Control
}

\titlerunning{Fairness via Dispersion Control}        

\author{Abhay Singh Bhadoriya \and Deepjyoti Deka \and Kaarthik Sundar
}

\authorrunning{A. Bhadoriya \and D. Deka \and K. Sundar} 

\institute{A. Bhadoriya \at
              \email{abhaysbh@amazon.com}           
           \and
           D. Deka \at
              MIT Energy Institute, MIT, Cambridge, MA.
              \email{deepj87@mit.edu} 
            \and 
          K. Sundar \at 
          Los Alamos National Laboratory, Los Alamos, NM.
          \email{kaarthik@lanl.gov}
}

\date{Received: date / Accepted: date}

\maketitle

\begin{abstract}
Controlling the dispersion of a subset of decision variables in an optimization problem is crucial for enforcing fairness or load-balancing across a wide range of applications. Building on the well-known equivalence of finite-dimensional norms, the article develops a family of parameterized convex models that regulate the dispersion of a vector of decision-variable values through its coefficient of variation. Each model has a single parameter taking values in the interval $[0,1]$. When the parameter is set to zero, the model imposes only a trivial constraint on the optimization problem; when set to one, it enforces equality of all the decision variables. As the parameter varies, the coefficient of variation is provably bounded above by a monotonic function of that parameter. The article also presents theoretical results relating the space of feasible solutions across all models. 
Finally, it compares the models' solution quality on a variant of the assignment problem that regulates the dispersion in the assignment costs.

\end{abstract}

\section{Introduction}
\label{sec:introduction}
Fairness has deep roots in game theory, where foundational bargaining solutions were rigorously characterized via axiomatic frameworks. Nash's seminal work formalized the bargaining problem through a set of axioms--Pareto efficiency, symmetry, invariance to affine transformations, and independence of irrelevant alternatives--which uniquely identify the Nash bargaining solution \cite{nash1950bargaining}. Subsequent refinements, such as the Kalai-Smorodinsky solution, altered these axioms to prioritize proportionality among agents \cite{kalai1977proportional}. In economics and communication theory \cite{xinying2023guide,shakkottai2008network}, these game-theoretic fairness axioms guide the design of utility functions so that decentralized, self-interested agents naturally converge to equitable outcomes. By carefully embedding fairness criteria into each agent's payoff, one ensures that the collective solution respects core notions of equity and parity, even when agents act independently.

However, in many application domains, problems are centrally optimized: a single decision maker minimizes a global cost or maximizes overall efficiency while imposing load-balancing requirements to distribute burdens--battery discharge, travel distance, computational load \cite{ghomi2017load,janiar2021deep,bhadoriya2024equitable}--uniformly across agents. Here, ``fairness'' is more accurately understood as dispersion control, in which the vector of decision-variable values must exhibit limited spread under a centralized objective.
In contemporary optimization, efforts have been made to extend beyond achieving minimal cost or maximal utility by explicitly regulating the dispersion of selected decision variables. By bounding or penalizing the spread through metrics such as variance, range, coefficient of variation, or more sophisticated equity indices \cite{xinying2023guide}, one can embed fairness or load-balancing constraints directly into existing optimization models. However, many dispersion metrics do not admit convex reformulations; directly bounding them thus renders the optimization problem non‑convex and computationally challenging \cite{xinying2023guide}. Nevertheless, the dispersion‑control paradigm is ubiquitous: ensuring equitable service levels in cloud-computing clusters \cite{ghomi2017load}, allocating shared spectrum in wireless networks \cite{matyjas2016spectrum}, operating power grids during extreme events \cite{sundar2025parametric}, among others. Controlling dispersion thus provides a unifying formalism that couples classical efficiency criteria with societally critical notions of parity, equity, and fairness.

\subsection{Literature Review} \label{subsec:lit-review}
In the literature, two principal modeling paradigms have emerged for incorporating fairness into decision-making problems that traditionally focus on maximizing efficiency.  

The first paradigm replaces the efficiency objective with an explicitly fairness-oriented objective--for example, the Rawlsian max–min objective \cite{rawls2017theory}, $\alpha$-fairness \cite{mo2002fair} (a concave objective with $\alpha$ as a parameter), proportional fairness \cite{shakkottai2008network} (a nonlinear objective), or $\ell_p$-norm dispersion objectives \cite{bektacs2020using} (a convex objective with $p$ as a parameter). In this framework, parity is pursued as the primary goal. However, a well-documented limitation \cite{sundar2025parametric} is that the parameters--$\alpha$ in the case of $\alpha$-fairness and $p$ in the case of $\ell_p$-norm objectives--do not have a direct or monotonic relationship with the fairness level achieved in the optimal solution. In other words, adjusting these parameters does not necessarily increase or decrease fairness.  

Fairness levels are typically assessed a posteriori using indices such as the Gini coefficient \cite{gini1936measure}, Atkinson index \cite{atkinson1970measurement}, Hoover index \cite{long1997hoover}, and Jain's index \cite{jain1984quantitative}, etc. A substantial body of work has sought to enforce fairness by directly minimizing or maximizing such indices, or by reducing dispersion measures such as the coefficient of variation or the range of decision variables \cite{xinying2023guide}. These approaches, however, generally suffer from one of two drawbacks: (i) the resulting model is non-convex, or (ii) the model produces a single fair solution that compromises efficiency--an outcome often unacceptable in applications where cost considerations are critical. Despite these issues, this paradigm remains widely used, given the limited availability of alternative frameworks for fairness in centralized optimization.  

The second, more recent paradigm--which motivates the present study--retains an efficiency-centric objective while constraining dispersion. Examples include the $\varepsilon$-fairness second-order cone approach \cite{sundar2025parametric} or the imposition of explicit upper bounds on fairness indices such as the Gini coefficient \cite{bektacs2020using},  Jain’s index, coefficient of variation, or range \cite{xinying2023guide}. Here, fairness is modeled by enforcing a lower (or upper) bound on the minimum (or maximum) acceptable fairness level while preserving the efficiency objective. This approach does not always yield a convex optimization problem, as noted in \cite{xinying2023guide}, and even when convexity is preserved, computing a suitable lower bound for fairness often requires trial-and-error. Two notable exceptions are the $\varepsilon$-fairness second-order cone method \cite{sundar2025parametric}, which achieves this using a single second-order cone constraint, and the direct bounding of the Gini coefficient \cite{bhadoriya2024equitable}, which reformulates the problem into a set of linear inequalities at the cost of introducing $O(n^2)$ auxiliary variables. For a comprehensive review of both paradigms, see the survey in \cite{xinying2023guide}. This article examines $\varepsilon$-fairness and extends it to a family of convex models. The appeal of $\varepsilon$-fairness lies in its desirable properties: it enforces fairness through a single parametric convex constraint with $\varepsilon \in [0,1]$ as the parameter, and the parameter $\varepsilon$ is monotonic in the level of fairness. Moreover, it naturally facilitates the study of the Pareto front of optimal solutions balancing efficiency and fairness in a bi-objective setting \cite{bhadoriya2024equitable}.

We now present the main contributions of this article in the context of this literature.  

\subsection{Contributions} \label{subsec:contributions}
The contributions of the article are as follows: we extend the recently introduced model of $\varepsilon$-fairness \cite{sundar2025parametric}--that controls the dispersion of decision variables via its coefficient of variation \cite{everitt2010cambridge}--to a countably infinite family of convex models. Each model is derived from the first principles through the classical equivalence of finite-dimensional norms \cite{halmos2017finite}, namely between the $\ell_p$-norm and the $\ell_1$-norm for any $p \geqslant 2$. All models are governed by a single parameter, $\varepsilon \in [0,1]$: when $\varepsilon = 0$, no constraint on the dispersion is imposed, and when $\varepsilon = 1$, the constraint forces the dispersion to be zero. 
In particular, for $p = 2$, the constraint reduces to a second-order cone \cite{sundar2025parametric}; for $p = \infty$, it reduces to a set of linear inequalities; and for $p \in (2, \infty)$, it reduces to the $\ell_p$-norm cone. In addition, we present theoretical results that characterize the relationships among these convex models for all values of $\varepsilon$, including their relative strength, inclusion properties, and points of equivalence.
Finally, we compare and contrast the models in terms of computational time and their effectiveness on a variant of the assignment problem in which the dispersion of assignment costs across agents is controlled. 


\section{Model Development} \label{sec:model}
We begin by introducing notation and recalling a classical result on the equivalence of norms in finite-dimensional spaces \cite{halmos2017finite}.  Let $\bm x \in \mathbb{R}^n_{\geqslant 0}$ denote an $n$-dimensional non-negative vector, whose dispersion we wish to regulate.  For any integer $p\geqslant 1$, define its $\ell^p$‐norm by
\begin{gather}
    \|\bm x\|_p \;=\;\Bigl(\sum_{i=1}^n x_i^p\Bigr)^{1/p}, \label{eq:p-norm}
\end{gather}
Here we omit the absolute value in \eqref{eq:p-norm} 
since $x_i \geqslant 0$ for all~$i$. Also define
\begin{gather}
    D_p \;\triangleq\; n^{1-1/p} - 1. \label{eq:D_r}
\end{gather}
Then, for any pair of integers $p_1$ and $p_2$, such that $0 < p_1 < p_2$, the norm‐equivalence result \cite{halmos2017finite} states
\begin{gather}
    \mathcal{NE}(p_1,p_2): \;\;\|\bm x\|_{p_2} \;\leqslant\; \|\bm x\|_{p_1} 
    \;\leqslant\;\left(\frac{D_{p_2} + 1}{D_{p_1} + 1}\right)\,\|\bm x\|_{p_2}. \label{eq:equivalence}
\end{gather}
Specializing to the case $(p_1,p_2)=(1,p)$ with $p\geqslant 2$, we obtain
\begin{gather}
    \mathcal{NE}(1,p): \;\;\|\bm x\|_p \;\leqslant\; \|\bm x\|_1 
    \;\leqslant\;(D_p + 1)\,\|\bm x\|_p. \label{eq:ne-1p}
\end{gather}
Moreover:
\begin{itemize}
  \item The lower bound $\|\bm x\|_p = \|\bm x\|_1$ is attained if and only if all but one component of $\bm x$ are zero (the most \emph{unfair} distribution of $\bm x$).
  \item The upper bound $\|\bm x\|_1 = (D_p + 1)\,\|\bm x\|_p$ is attained if and only if all components of $\bm x$ are equal (the most \emph{fair} distribution of $\bm x$).
\end{itemize}
This intuition naturally leads to the subsequent Definition \ref{defn:ep-fairness}, extending the concept of $\varepsilon$-fairness in \cite{sundar2025parametric}. To keep the presentation self-contained, we first recall the notion of $\varepsilon$-fairness introduced in \cite{sundar2025parametric}. To that end, for any $\bm x \in \mathbb{R}^n_{\geqslant 0}$ and $\varepsilon \in [0,1]$, $\bm x$ is said to be \emph{$\varepsilon$-fair} if it satisfies
\begin{gather}
(1 + \varepsilon D_2)\,\|\bm x\|_2 = \|\bm x\|_1 \label{eq:eps-fair}
\end{gather}
and $\bm x$ is \emph{at least $\varepsilon$-fair} if 
\begin{gather}
(1 + \varepsilon D_2)\,\|\bm x\|_2 \leqslant \|\bm x\|_1. \label{eq:eps-fair-soc}
\end{gather}
The constraint \eqref{eq:eps-fair-soc} is a second-order cone with larger values of $\varepsilon$ corresponding to greater fairness. The following Definition \ref{defn:ep-fairness} generalizes the above construction by replacing the $\ell_2$-norm with $\ell_p$-norms for any $p \geqslant 2$.
\begin{definition} \label{defn:ep-fairness}
$(\varepsilon,p)$‐fairness -- For any integer $p\geqslant 2$, $\bm x \in \mathbb{R}^n_{\geqslant 0}$ and $\varepsilon\in[0,1]$, we say that $\bm x$ is \emph{$(\varepsilon,p)$‐fair} if
\begin{gather*}
    \bigl(1 + \varepsilon\,D_p\bigr)\,\|\bm x\|_p \;=\;\|\bm x\|_1.
\end{gather*}
Furthermore, we say that $\bm x$ is \emph{at least $(\varepsilon,p)$‐fair} if
\begin{gather}
    \bigl(1 + \varepsilon\,D_p\bigr)\,\|\bm x\|_p 
    \;\leqslant\;\|\bm x\|_1. \label{eq:e-fair}
\end{gather}
Additionally, we denote the set of at least $(\varepsilon,p)$-fair vectors as
\begin{gather}
        \mathcal X(\varepsilon, p)
    \;\triangleq\;
    \bigl\{\bm x \in \mathbb{R}_{\geqslant 0}^n : 
    (1 + \varepsilon\,D_p)\,\|\bm x\|_p \leqslant \|\bm x\|_1 \bigr\}. \label{eq:X}
\end{gather}
\end{definition}
In the above definition, the factor $(1 + \varepsilon D_p)$ can be interpreted as the convex combination of $1$ and $(D_p + 1)$, with $\varepsilon$ as the non-negative multiplier. Furthermore, the constraint in \eqref{eq:e-fair} is convex for every $p \geqslant 2$.  In particular, for $p = 2$, \eqref{eq:e-fair} reduces to a second‐order cone \cite{sundar2025parametric}; for $2 < p < \infty$, it defines an $\ell_p$‐norm cone, and for $p = \infty$, it reduces to a collection of linear inequalities as follows:
\begin{flalign}
    \|\bm x\|_1 \;\geqslant\;(1+\varepsilon D_\infty)\,\|\bm x\|_\infty  \iff
    \sum_{i=1}^n x_i \;\geqslant\;\left(1+n\varepsilon -\varepsilon\right)\cdot x_i
    \quad\forall\,i.
     \label{eq:linear}
\end{flalign}
For any fixed $\varepsilon$, the Definition~\ref{defn:ep-fairness} yields a countably infinite family of convex models (by varying the integer $p \geqslant 2$) for enforcing fairness in the distribution of~$\bm x$.
\begin{remark}[Connection to $\varepsilon$-fairness]
For $p = 2$, Definition \ref{defn:ep-fairness} reduces exactly to the $\varepsilon$-fairness model introduced in \cite{sundar2025parametric}.  Thus, the $(\varepsilon,p)$-fairness framework can be viewed as a direct generalization of $\varepsilon$-fairness to a family of models indexed by $p \geqslant 2$.  While the case $p=2$ yields a second-order cone constraint, the general formulation allows for different convex representations: an $\ell_p$-norm cone for $2 < p < \infty$ and a set of linear inequalities for $p = \infty$. This generalization preserves convexity while enabling a broader range of modeling and computational trade-offs.
\end{remark}
We now present a proposition that directly follows from the definition of $\mathcal X(\varepsilon, p)$ in \eqref{eq:X}. 
\begin{proposition} \label{prop:eps-varying}
    For any integer $p \geqslant 2$ and and $\varepsilon_1, \varepsilon_2 \in [0, 1]$ such that $\varepsilon_1 > \varepsilon_2$, $\mathcal X( \varepsilon_1,p) \subset \mathcal X( \varepsilon_2,p)$. \hfill \qed
\end{proposition}
The following normalization argument to simplifies notation for the main theoretical results in the subsequent section.  
\begin{remark} \label{rem:norm}
{\it Normalization} -- Since all finite-dimensional norms are positively homogeneous of degree 1, the inequality in \eqref{eq:X}, i.e., 
\begin{gather*}
    \bigl(1 + \varepsilon\,D_p\bigr)\,\|\bm x\|_p 
    \;\leqslant\;\|\bm x\|_1
\end{gather*}
is invariant under scaling of $\bm x$.  Thus, it is convenient to normalize any non-zero vector\footnote{The zero vector, $\bm 0$, cannot be normalized and is trivially $(\varepsilon,p)$-fair for any $p \geqslant 2$ and $\varepsilon \in [0, 1]$. Hence, throughout the rest of the article, we assume that $\bm x$ is always a non-zero vector.} by its $\ell_1$-norm.  To that end, define
\begin{gather}
  \Delta_n \;\triangleq\;\left\{\bm x\in\mathbb R_{\geqslant 0}^n : \sum_{i=1}^n x_i = 1\right\},
  \label{eq:prob-simplex}
\end{gather}
the standard probability simplex in $\mathbb R^n$.  Normalizing any $\bm x \in \mathcal X(\varepsilon, p)$ with respect to its $\ell_1$-norm then yields  
\begin{gather}
  \mathcal Y(\varepsilon,p)
  \;\triangleq\;
  \left\{\bm x\in\Delta_n : (1+\varepsilon\,D_p)\,\|\bm x\|_p \leqslant 1\right\}.
  \label{eq:X1}
\end{gather}
Observe that $\mathcal Y(\varepsilon,p)$ is exactly the image of $\mathcal X(\varepsilon,p)$ under the map $\bm x\mapsto\bm x/(\bm e^\intercal\bm x)$; here $\bm e$ is a vector of 1s. Hence, any theoretical result for $\mathcal X(\varepsilon,p)$ can be transferred immediately to $\mathcal Y(\varepsilon,p)$ and vice-versa. All theoretical properties of $(\varepsilon, p)$-fairness in the subsequent section will be proved for the normalized set $\mathcal Y(\varepsilon, p)$. 
\end{remark}

\section{Theoretical Properties} \label{sec:theory}
The results in this section establish the validity and structure of the proposed $(\varepsilon,p)$-fairness models from three complementary perspectives. First, the Section \ref{subsec:cv} shows that the constraint has a clear statistical interpretation by deriving an explicit upper bound on the coefficient of variation, thereby justifying its role as a dispersion-control (fairness) mechanism. Next, the Section \ref{subsec:equiv} characterizes the behavior at the boundary values $\varepsilon \in \{0,1\}$, showing that all models in the family coincide at these extremes. Finally, the Section \ref{subsec:strict} focuses on the interior regime $\varepsilon \in (0,1)$ and establishes a strict inclusion relationship across different values of $p$, revealing how the choice of $p$ affects the feasible region. Together, these results provide a coherent picture of how the parameter $\varepsilon$ controls dispersion and how the family of models is structured across different values of $p$.

We first establish the relationship between the $(\varepsilon,p)$-fairness of Definition~\ref{defn:ep-fairness} and the coefficient of variation of the vector~$\bm x$. The coefficient of variation of $\bm x$ is a statistical measure of the dispersion of $\bm x$. The result in the subsequent Section \ref{subsec:cv} provides the intuition for how varying $\varepsilon$ in \eqref{eq:e-fair} for any fixed integer $p \geqslant 2$ allows the control of dispersion of $\bm x$.

\subsection{Bound of Coefficient of Variation} \label{subsec:cv}
The first property we establish is that the $(\varepsilon,p)$‐fairness constraint \eqref{eq:e-fair} implies an upper bound on the coefficient of variation of $\bm x$, and that this bound tightens as $\varepsilon$ increases.
\begin{proposition} \label{prop:cv-bound}
For any $\varepsilon \in [0, 1]$, an integer $p \geqslant 2$, and $\bm x \in \mathcal Y(\varepsilon, p)$ defined in \eqref{eq:X1}, the coefficient of variation of $\bm x$
\begin{gather*}
    \mathrm{CV}(\bm x)
    \triangleq \frac{\sqrt{\frac{1}{n}\sum_{i=1}^n (x_i - \mu)^2}}{\mu},
    \quad \text{ where } \quad
    \mu = \frac{1}{n},
\end{gather*}
satisfies the following bound:
\begin{gather}
    \left(\mathrm{CV}(\bm x)\right)^2
    \;\leqslant\;
    B_p(\varepsilon)
    \;\triangleq\;
    \frac{(D_p+1)^2}{(1 + \varepsilon\,D_p)^2} - 1\,.
    \label{eq:cv-bound-p}
\end{gather}
Moreover, $B_p(\varepsilon)$ is strictly decreasing in $\varepsilon$ on $[0,1]$, with  $B_p(0)=(D_p+1)^2-1$ and $B_p(1)=0$.
\end{proposition}
\begin{proof}
For any $p \geqslant 2$, using norm equivalence $\mathcal{NE}(2, p)$ in \eqref{eq:equivalence} and \eqref{eq:X1}, we have
\begin{gather}
    \|\bm x\|_2 \;\leqslant\;\frac{D_p + 1}{\sqrt n}\|\bm x\|_p \;\leqslant \;\frac{D_p + 1}{\sqrt n}\left(\frac{1}{1 + \varepsilon\,D_p}\right). \label{eq:2-1-bound}
\end{gather}
On the other hand, since $n\mu = 1$ and observing that
\begin{gather*}
    \frac{1}{n}\sum_{i=1}^n x_i^2
    = \frac{\|\bm x\|_2^2}{n},
    \quad
    \frac{1}{n}\sum_{i=1}^n (x_i-\mu)^2
    = \frac{\|\bm x\|_2^2}{n} - \mu^2,
\end{gather*}
we obtain
\begin{gather*}
    \left(\mathrm{CV}(\bm x)\right)^2
    = \frac{\|\bm x\|_2^2 - n\mu^2}{n\mu^2}
    = n\,\|\bm x\|_2^2 - 1.
\end{gather*}
Substituting the bound on $\|\bm x\|_2$ from \eqref{eq:2-1-bound} yields
\begin{flalign*}
    n\,\|\bm x\|_2^2 - 1
    &\;\leqslant\;
    \frac{(D_p+1)^2}{(1 + \varepsilon\,D_p)^2}
    - 1.
\end{flalign*}
Hence,
\begin{gather*}
    \left(\mathrm{CV}(\bm x)\right)^2
    \;\leqslant\;
    \frac{(D_p+1)^2}{(1 + \varepsilon\,D_p)^2} - 1
    = B_p(\varepsilon).
\end{gather*}
Finally, since $D_p>0$, we have
\begin{gather*}
    B_p'(\varepsilon)
    \;=\;
    -\,\frac{2D_p(D_p+1)^2}{(1 + \varepsilon\,D_p)^3}
    \;<\;0.
\end{gather*}
Hence, $B_p(\varepsilon)$ is strictly decreasing in $\varepsilon$ and $B_p(0)=(D_p+1)^2-1$, $B_p(1)=0$, as claimed.
\end{proof}

The strict monotonicity of $B_p(\varepsilon)$ in $\varepsilon$ implies that, as $\varepsilon$ increases, the upper bound on $\mathrm{CV}(\bm x)$ decreases, thereby decreasing the dispersion of~$\bm x$. This relationship was previously established for the case $p=2$ in \cite{sundar2025parametric}.

In the forthcoming sections, we
present the main theoretical results that concern the equivalence or ordered relation between sets $\mathcal Y(p, \varepsilon)$ and correspondingly of $\mathcal X(p, \varepsilon)$ for varying values of $p$. 

\subsection{Equivalence Condition} \label{subsec:equiv}
In this section, we give conditions under which requiring $\bm x$ to be at least $(\varepsilon,p)$\nobreakdash‑fair is equivalent for every integer $p\geqslant 2$.  
We start with a key result that directly follows from the definition of $(\varepsilon, p)$-fairness. 
\begin{proposition} \label{prop:corner}
    Suppose that $\bm e_i$ is a unit vector with the $i$\textsuperscript{th} coordinate set to $1$ and $\bm e$ is a vector of 1s, then for any integer $p \geqslant 2$, the following statements hold:
    \begin{gather}
        \text{For any } i \in \{1, \dots, n\},~~ \bm e_i \in \mathcal Y(\varepsilon, p) ~\Longrightarrow ~ \varepsilon = 0 \label{eq:lem-a} \\
        \text{ and } \varepsilon = 1 \iff \mathcal Y(\varepsilon, p) = \left \{ \tfrac 1n \bm e \right\}. \label{eq:lem-b} 
    \end{gather}
\end{proposition}
\begin{proof}
    The proof follows from the observation in Section \ref{sec:model} that the term $(1 + \varepsilon D_p)$ is the convex combination of the two scalars $1$ and $(D_p + 1)$ with $\varepsilon$ as the factor.
\end{proof}
\begin{theorem} \label{thm:equiv}
For any integer $p \geqslant 2$, and $\varepsilon\in\{0,1\}$, the sets of at least $(\varepsilon,p)$-fair vectors are equivalent, i.e.
\[
    \mathcal Y(\varepsilon,p)
    = \mathcal Y(\varepsilon,2)
    \quad\forall~ p > 2.
\]
\end{theorem}
\begin{proof}
We treat the two cases $\varepsilon=0$ and $\varepsilon=1$ separately. When $\varepsilon = 0$, the inequality in \eqref{eq:X1} is trivially satisfied for any $\bm x \in \Delta_n$ and any $p \geqslant 2$. Hence, 
\[
    \mathcal Y(0,p) = \Delta_n = \mathcal Y(0,2) \quad \forall, p > 2.
\]
When $\varepsilon = 1$, Proposition \ref{prop:corner} indicates that  
$\mathcal Y(1,p) = \left\{\tfrac 1n \bm e \right\}$ for every $p \geqslant 2$.
\end{proof}
Theorem \ref{thm:equiv} shows that at the two extreme fairness levels, the feasible set of vectors $\bm x$ is the same regardless of the choice of $p$.

\subsection{Strict Inclusion Property} \label{subsec:strict}
Here, we prove a result that shows a strict inclusion property between $\mathcal Y(\varepsilon, p_1)$ and $\mathcal Y(\varepsilon, p_2)$ for any pair of integers $(p_1, p_2)$ such that $2 \leqslant p_1 < p_2$, and $\varepsilon$ in the open interval $(0, 1)$. We begin by presenting a proposition that will be useful to prove the main theorem.
\begin{proposition} \label{prop:decreasing}
    For any $\bm x \in \Delta_n \setminus \left\{ \bm e_1, \dots, \bm e_n, \tfrac{1}{n}\bm e\right\}$,
    \begin{gather} 
        F(p) \triangleq \frac{1-\|\bm x\|_p}{D_p \|\bm x\|_p}  \text{ is strictly decreasing in } p \geqslant 2.
    \end{gather}
\end{proposition}
\begin{proof}
    See Appendix \ref{app:proof}. 
\end{proof}
\begin{theorem} \label{thm:inclusion}
    For any pair of integers $(p_1, p_2)$ such that $2 \leqslant p_1 < p_2$ and $\varepsilon \in (0, 1)$, $\mathcal Y(\varepsilon, p_2) \subset \mathcal Y(\varepsilon, p_1)$. 
\end{theorem}
\begin{proof}
    For each $\bm x\in\Delta_n$ and integer $p\geqslant 2$, define
    \begin{gather}
    \varepsilon_p(\bm x)
    \;\triangleq\;
    \frac{1-\|\bm x\|_p}{D_p\,\|\bm x\|_p}.
    \label{eq:eps-p}
    \end{gather}
    Since
    \[ (1+\varepsilon\,D_p)\,\|\bm x\|_p\leqslant 1   \iff \varepsilon\leqslant \varepsilon_p(\bm x) \]
       we have
    \[ \bm x\in\mathcal Y(\varepsilon,p) \iff 0<\varepsilon\leqslant \varepsilon_p(\bm x).\]
    Therefore $\mathcal Y(\varepsilon,p_2)\subset\mathcal Y(\varepsilon,p_1)$ will follow once we show
    \begin{gather}
    \varepsilon_{p_1}(\bm x) > \varepsilon_{p_2}(\bm x) ~~ \forall ~~ \bm x \in \Delta_n \setminus \left\{ \bm e_1, \dots, \bm e_n, \tfrac{1}{n}\bm e\right\}. \label{eq:final-claim}
    \end{gather}
   The excluded cases correspond to the boundary situation $\varepsilon \in \{0, 1\}$, already treated in Proposition \ref{prop:corner} and Theorem \ref{thm:equiv}.
   For any such $\bm x \in \Delta_n \setminus \left\{ \bm e_1, \dots, \bm e_n, \tfrac{1}{n}\bm e\right\}$, the proof is complete if we show that the map 
    \[ 
        p\mapsto\varepsilon_p(\bm x) = \frac{1-\|\bm x\|_p}{D_p \|\bm x\|_p}
    \]
    is \emph{strictly decreasing} on $[2,\infty]$ which is exactly the statement of Proposition \ref{prop:decreasing}. Figure \ref{fig:proof} illustrates the idea graphically.
   \begin{figure}[htbp]
       \centering
       \includegraphics[scale=0.7]{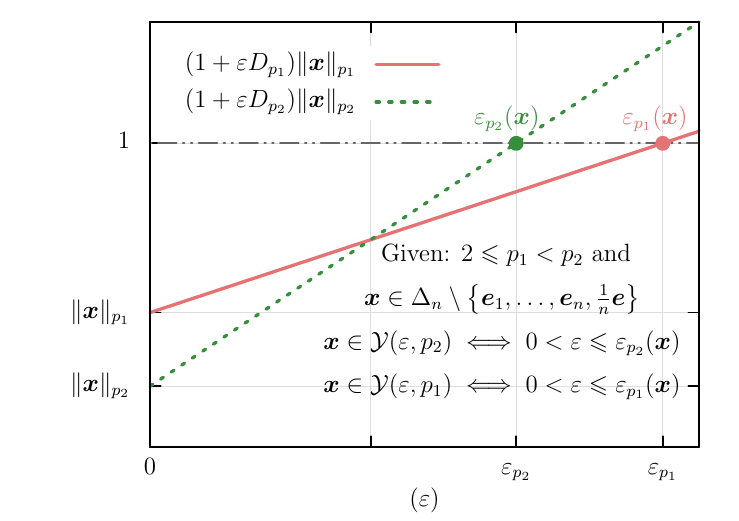}
       \caption{Graphical illustration for $\varepsilon_{p_1}(\bm x) > \varepsilon_{p_2}(\bm x)$.}
       \label{fig:proof}
   \end{figure}
\end{proof}

Theorem \ref{thm:inclusion} shows that the size of the feasible sets of requiring $\bm x$ to be at least $(\varepsilon, p)$-fair decreases with $p$ when $\varepsilon \in (0, 1)$ and Theorem \ref{thm:equiv} shows that the feasible sets are equal for any value of integer $p \geqslant 2$ when $\varepsilon \in \{0, 1\}$. 
Next, we consider an example application to illustrate the use of the models presented thus far in this article.

\section{Multi-Agent Assignment Problem} \label{sec:gap}
In this section, we formulate a basic version of a multi-agent assignment problem (MAAP) and incorporate the fairness models introduced in Section \ref{sec:model} to control the dispersion of the total time required to complete the tasks across different agents. We start by introducing the relevant notation. Let $\mathcal A$ denote the set of agents and $\mathcal J$ the set of tasks. We assume that $|\mathcal J| > |\mathcal A|$. Each task $j\in\mathcal J$ must be completed by exactly one agent. For any agent $a\in\mathcal A$ and task $j\in\mathcal J$, let $t_{a,j}$ denote the time taken by agent $a$ to complete task $j$. The multi-agent assignment problem seeks to assign each task $j\in\mathcal J$ to an agent in $\mathcal A$ to minimize the total completion time. Furthermore, we introduce binary variables
\[
x_{a,j} =
\begin{cases}
1, & \text{if task }j\text{ is assigned to agent }a,\\
0, & \text{otherwise.}
\end{cases}
\]
Using these binary decision variables, a Mixed-Integer Linear Program (MILP) for the MAAP is as follows:
\begin{subequations}
    \begin{flalign}
        \min: ~~& \sum_{a \in \mathcal A} t_a \\ 
        \text{subject to:} ~~& t_a = \sum_{j \in \mathcal J} t_{a,j}\cdot x_{a,j} 
        ~~ \forall a \in \mathcal A, \\ 
        & \sum_{a \in \mathcal A} x_{a,j} = 1 
        ~~ \forall j \in \mathcal J,\\ 
        & x_{a,j} \in \{0, 1\} 
        ~~\forall a \in \mathcal A,\;j \in \mathcal J.
    \end{flalign}
\label{eq:maap}
\end{subequations}
Although the formulation in \eqref{eq:maap} is presented as an MILP, its constraint matrix is the node–arc incidence matrix of a bipartite graph, which is totally unimodular. Consequently, the linear programming relaxation has integral extreme points, and solving it via any linear programming solver yields the optimal integer solution \cite{schrijver2003combinatorial}. Equivalently, one can apply the Hungarian algorithm directly to obtain the optimum in $O(|\mathcal J|^3)$ time \cite{kuhn1955hungarian}. In contrast, a natural generalization in which each task carries a weight and each agent has a capacity constraint becomes NP‑hard (the generalized assignment problem) \cite{cattrysse1992survey}. In this work, we restrict attention to the uncapacitated variant in \eqref{eq:maap}, focusing on illustrating how $(\varepsilon,p)$‑fairness models control the dispersion of the processing times $t_a$ across agents in the simplest assignment problem, rather than a general variant of the MAAP.

\subsection{Fairness-Constrained Multi-Agent Assignment Problem} \label{subsec:fair-maap}
We let $\bm t$ denote the vector of $t_a$ values over all agents $a \in \mathcal A$. The fairness‑constrained MAAP seeks to control the dispersion of the vector $\bm t$.  This is enforced by requiring $\bm t$ to be at least $(\varepsilon,p)$‑fair for any $\varepsilon\in(0,1]$ and $p\geqslant2$, as in \eqref{eq:e-fair}. When $\varepsilon=0$, constraint \eqref{eq:e-fair} becomes trivial and the problem reduces to \eqref{eq:maap}.  The fairness constraint for fixed $p\geqslant2$ and $\varepsilon\in(0,1]$ is  
\begin{gather}
    (1 + \varepsilon D_p)\,\|\bm t\|_p \;\leqslant\;\sum_{a\in\mathcal A} t_a.
    \label{eq:fairness-constraint-maap}
\end{gather}
Although the problem in \eqref{eq:maap} is solvable in polynomial time, adding \eqref{eq:fairness-constraint-maap} makes the resulting problem NP‑complete.  Indeed, for two agents with $\varepsilon=1$, it reduces to the classical partition problem \cite{korf1998complete} that is known to be strongly NP-complete. 
In the original $\varepsilon$-fairness framework of \cite{sundar2025parametric}, the constraint \eqref{eq:fairness-constraint-maap} for $p=2$ is implemented as a single second-order cone constraint, leading to a mixed-integer second-order cone program that can be handled directly by conic solvers. 
As mentioned previously in Section \ref{sec:model}, the present work extends this formulation to arbitrary $p \geqslant 2$. For $2 < p < \infty$, the fairness constraint becomes an $\ell_p$-norm cone, which we handle via cone disaggregation and outer approximation (Section~\ref{subsec:oa}). For $p = \infty$, it reduces to a set of linear inequalities, yielding a mixed-integer linear program. Thus, while the overall modeling framework and outer-approximation approach are inherited from prior work, the extension to a parametric family of cones and their comparative analysis are new.


\subsection{Outer Approximation of $\ell_p$-Norm Cone} \label{subsec:oa}
We start with the definition of an $\ell_p$-norm cone \cite{boyd2004convex} of dimension $(n+1)$ as
\begin{gather}
    P = \left\{(z, \bm y)\in\mathbb R^{n+1}: z \geqslant \|\bm y\|_p\right\}.
    \label{eq:p-norm-cone}
\end{gather}
Note that the fairness constraint in \eqref{eq:fairness-constraint-maap} for any $p\geqslant2$ can be represented in the form \eqref{eq:p-norm-cone} by introducing an auxiliary variable.  It is known in the literature that when convex cones are outer approximated by linear inequalities, it is both theoretically stronger and computationally more efficient to decompose a high-dimensional cone into several lower-dimensional cones and approximate each of those individually, rather than approximating the high-dimensional cone directly \cite{lubin2018polyhedral}.  Converting a high-dimensional conic constraint into many low-dimensional cones is called ``cone disaggregation'' \cite{lubin2018polyhedral}.  To that end, we present the steps to disaggregate the $(n+1)$-dimensional $\ell_p$-norm cone in \eqref{eq:p-norm-cone} into $n$ three-dimensional cones plus a single linear constraint:
\begin{subequations}
\begin{gather*}
    z^p \;\geqslant\;\sum_{i=1}^n y_i^p
    ~~\equiv~~
    z \;\geqslant\;\sum_{i=1}^n \frac{y_i^p}{z^{p-1}},
    \label{eq:p-norm-cone-1}\\
    \equiv~~
    \left\{\,z = \sum_{i=1}^n Y_i \;\text{and}\; Y_i\,z^{p-1}\geqslant y_i^p~~\forall\,i\right\}\\
    \equiv~~
    \left\{\,z = \sum_{i=1}^n Y_i \;\text{and}\; Y_i^{\frac1p}\,z^{1-\frac1p}\geqslant y_i~~\forall\,i\right\}.
\end{gather*}
\end{subequations}
In the last equivalence, each constraint $Y_i^{\tfrac 1p}\,z^{1-\tfrac 1p}\geqslant y_i$ is convex for $2\leqslant p<\infty$ and in fact defines a three-dimensional power cone \cite{boyd2004convex}:
\begin{flalign}
    \mathcal P_3^{\alpha,1-\alpha}
    &\triangleq
    \left\{(v_1,v_2,v_3)\in\mathbb R_{\geqslant0}^3: v_1^{\alpha}v_2^{1-\alpha}\geqslant v_3\right\}.
    \label{eq:power-cone}
\end{flalign}
Hence, the cone in \eqref{eq:p-norm-cone} can be rewritten as
\begin{flalign}
    \left\{\,z = \sum_{i=1}^n Y_i \;\text{and}\; (Y_i,z,y_i)\in\mathcal P_3^{1/p,1-1/p}~~\forall\,i\right\},
    \label{eq:dis-p}
\end{flalign}
which comprises $n$ three-dimensional power cones and one linear equality.  A linear outer approximation of each power cone is obtained by the first-order Taylor expansion at a point $(Y_i^0,z^0,y_i^0)$ on its surface (i.e., satisfying the conic constraint at equality).  Specifically, the linear outer approximation of $(Y_i,z,y_i)\in\mathcal P_3^{1/p,1-1/p}$ at $(Y_i^0,z^0,y_i^0)$ is
\begin{flalign}
    \left(\frac1p\right)\,y_i^0\,z^0\,Y_i \;+\;\left(1-\frac1p\right)\,y_i^0\,Y_i^0\,z
    \;\geqslant\;
    Y_i^0\,z^0\,y_i.
    \label{eq:oa-p}
\end{flalign}
In the algorithm to solve the mixed-integer convex optimization problem, we relax the $n$ three-dimensional power cones and iteratively add these linear outer approximations whenever the solution to the relaxed problem violates one of them.  The required point on the surface of the power cone is computed by projecting the infeasible solution onto the corresponding power cone. 

\section{Computational Results} \label{sec:results}
This section provides empirical validation of the theoretical results in Section~\ref{sec:theory} using the formulations from Section~\ref{sec:gap}. All experiments were run on 50 randomly generated instances of the MAAP, each with 5 agents and 250 tasks; the processing times $t_{a,j}$ were drawn uniformly at random. Both variants--with and without fairness constraints--were implemented in Kotlin and solved using CPLEX~22.1. The outer‑approximation algorithm of Section~\ref{subsec:oa} was realized via CPLEX's lazy‑constraint callback. Source code for instance generation and both solution methods is available at  
\url{https://github.com/abhay-singh11/FairTaskAllocation}. 

For each instance we solved the base assignment problem \eqref{eq:maap} augmented by the \((\varepsilon,p)\)–fairness constraint \eqref{eq:fairness-constraint-maap} for  
\[
  p\in\{2,3,5,10,\infty\}, 
  \quad 
  \varepsilon\in\{0,0.1,\dots,0.9,0.99,1.0\}.
\]
We found that all 50 instances were infeasible at $\varepsilon=1.0$, indicating that that perfect fairness cannot be enforced for the instances; we attribute this behaviour to the fact that $\{t_{a, j}\}$ are generated randomly. For this reason, we include $\varepsilon=0.99$ in our study: at this level, all instances remain feasible for every chosen $p$.

We let $\bm t^*(\varepsilon,p)$ denote the the optimal solution vector for the MAAP under fixed $(\varepsilon,p)$, and $\bm t^*(0)$ for the solution of the unconstrained problem \eqref{eq:maap}. This notation is chosen to explicitly indicate that when $\varepsilon = 0$ in \eqref{eq:fairness-constraint-maap}, it is equivalent to imposing no fairness constraints on the MAAP.
 
\subsection{Dispersion Control} \label{subsec:dispersion-results}
Here, we present the results for one representative instance out of the 50 randomly generated ones, since the other 49 exhibit the same qualitative behaviour. Figure~\ref{fig:cv} reports the coefficient of variation $\mathrm{CV}(\bm t^*(\varepsilon, p))$, computed as in Proposition~\ref{prop:cv-bound}. First, we observe that $\mathrm{CV}(\bm t^*(\varepsilon, p))$--the measure of dispersion--decreases monotonically with $\varepsilon$. Proposition~\ref{prop:cv-bound} establishes that the upper bound on the coefficient of variation decreases in $\varepsilon$; Figure~\ref{fig:cv} confirms that the actual $\mathrm{CV}(\bm t^*(\varepsilon, p))$ likewise decreases. Next, we note the trend of diminishing returns: beyond $\varepsilon\approx0.70$, the marginal reduction in dispersion is negligible, indicating that least fairness gains occur at moderate values of $\varepsilon$.

\begin{figure}[htbp]
    \centering
    \includegraphics[scale=0.5]{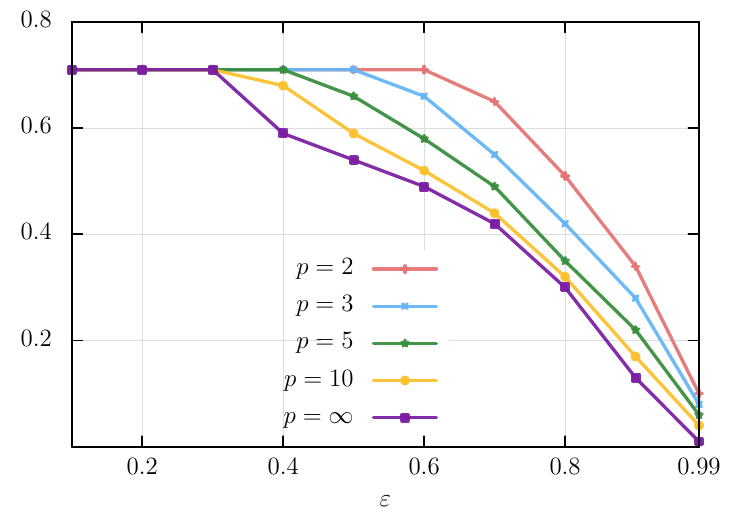}
    \caption{Coefficient of variation $\mathrm{CV}(\bm t^*(\varepsilon, p))$ versus $\varepsilon$ for different $p$ for one instance.}
    \label{fig:cv}
\end{figure}

\subsection{Price of Fairness} \label{subsec:pof}
The term ``Price of Fairness'' (PoF) is defined as the percentage increase in the objective value of the fairness‑constrained MAAP of Section~\ref{subsec:fair-maap} relative to that of the unconstrained problem \eqref{eq:maap}.  This definition is standard in the literature \cite{bertsimas2011price,sundar2025parametric}.  Formally,
\[
    \mathrm{PoF}(\varepsilon,p)
    \;\triangleq\;\left(\frac{\bm t^*(\varepsilon,p)^\intercal\bm e}{\bm t^*(0)^\intercal\bm e}-1\right)\,\times100\%.
    \label{eq:pof}
\]
\begin{figure}[htbp]
    \centering
    \includegraphics[scale=0.5]{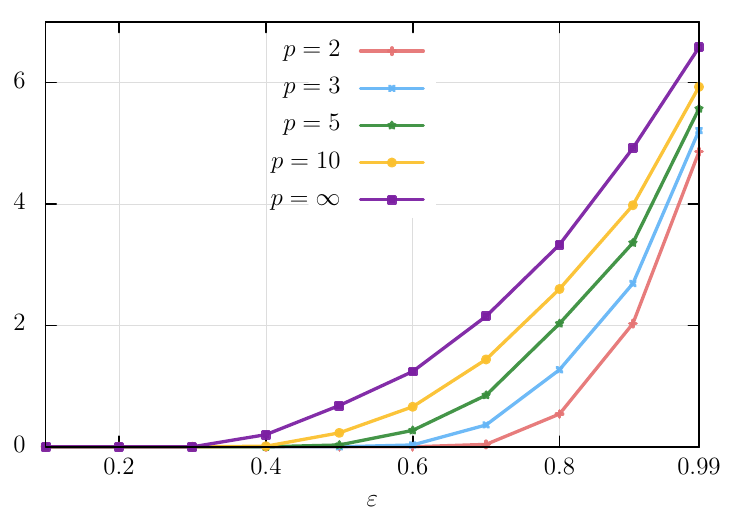}
    \caption{$\mathrm{PoF}(\varepsilon,p)$ versus $\varepsilon$ for various $p$.}
    \label{fig:pareto}
\end{figure}
Figure~\ref{fig:pareto} illustrates this trend for one representative instance; the other 49 instances show the same qualitative behavior.  The increasing ordering 
$\mathrm{PoF}(\varepsilon,p_1) \geqslant \mathrm{PoF}(\varepsilon,p_2)$ for $p_1>p_2$
reflects the inclusion property of Theorem~\ref{thm:inclusion}.  

\subsection{Practical Tuning of $\varepsilon$ and $p$} \label{subsec:tuning}
The parameters $(\varepsilon,p)$ play complementary roles in practice. For any fixed $p$, increasing $\varepsilon$ provides a continuous way to trace the fairness--efficiency trade-off, with larger values imposing tighter dispersion control. By contrast, $p$ determines the geometry of the dispersion measure and thus the shape of the attainable trade-off curve.  This distinction is also evident empirically. As shown in Figure~\ref{fig:pareto}, for larger values of $p$, varying $\varepsilon$ yields a more diverse set of solutions spanning a wider range of fairness--efficiency combinations, whereas smaller values of $p$ produce a more compressed frontier. 

These observations suggest a practical tuning strategy. One may first select $p$ based on modeling and computational considerations. For example, smaller values, such as $p=2$, yield smooth aggregate control. Larger values place greater emphasis on extreme deviations. The case $p=\infty$ admits a linear formulation. One can then tune $\varepsilon$ to select a desired operating point along the resulting trade-off curve.  In particular, increasing $\varepsilon$ is most natural when refining fairness within a fixed modeling framework. Increasing $p$ is preferable when a richer or more expressive set of candidate solutions is desired.

\subsection{Computation Times} \label{subsec:times}
The box‑plots of computation times are shown in Figure~\ref{fig:time}.  Two trends stand out: (i) The computation time increases with $\varepsilon$.  As $\varepsilon$ grows, the feasible region shrinks, making it harder for the solver to find and prove feasibility. (ii) For $p=\infty$, solve times are generally lower than for finite $p$.  In that case, the fairness constraint reduces to a set of linear inequalities, whereas for $p<\infty$ the conic constraints incur additional overhead. Furthermore, when $p < \infty$, computation times are observed to increase with $p$. 

\begin{figure*}[htb!]
\centering
  \begin{subfigure}[t]{.45\linewidth}
    \centering\includegraphics[scale=0.4]{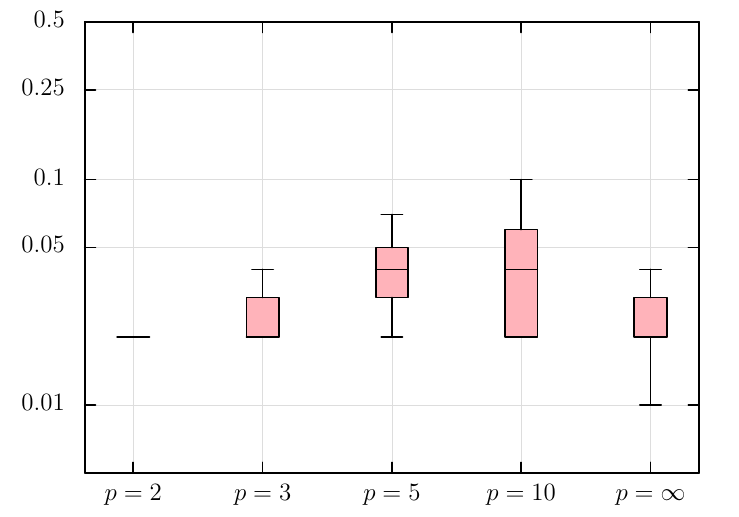}
    \caption{$\varepsilon = 0.5$}
  \end{subfigure} 
  \begin{subfigure}[t]{.45\linewidth}
    \centering\includegraphics[scale=0.4]{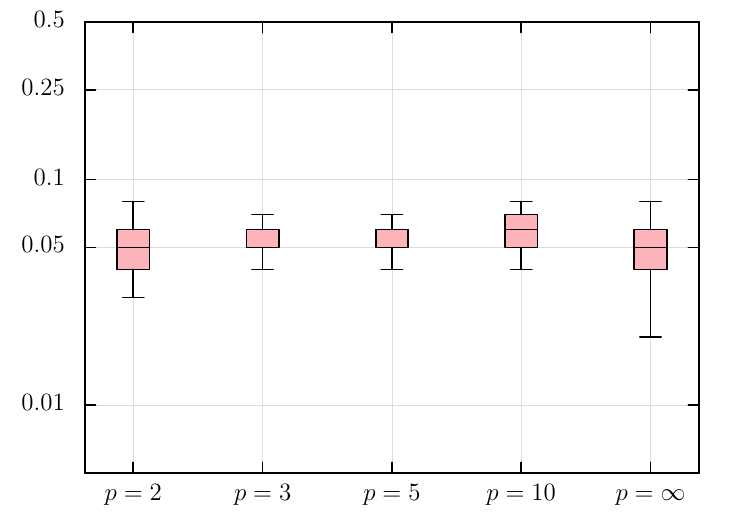}
    \caption{$\varepsilon = 0.7$}
  \end{subfigure} \\
  \begin{subfigure}[t]{0.45\linewidth}
\centering\includegraphics[scale=0.4]{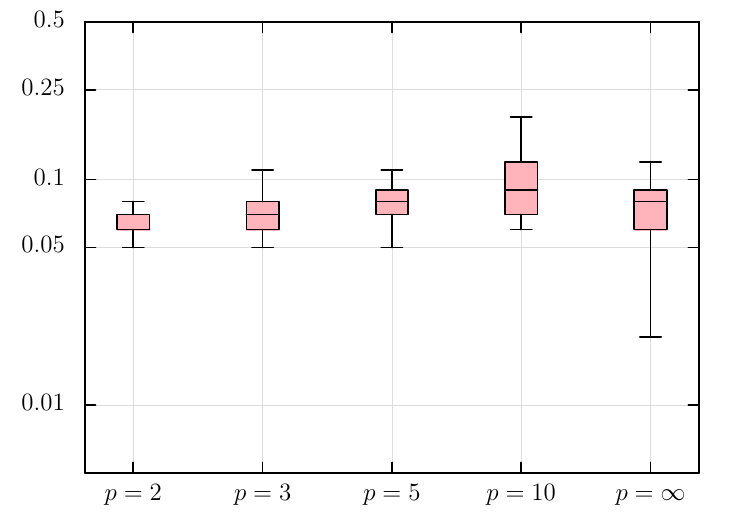}
    \caption{$\varepsilon = 0.9$}
  \end{subfigure} 
  \begin{subfigure}[t]{0.45\linewidth}
    \centering\includegraphics[scale=0.4]{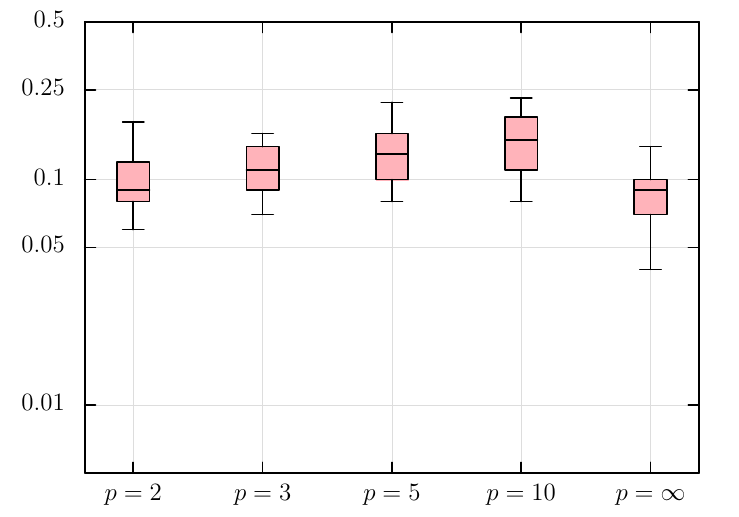}
    \caption{$\varepsilon = 0.99$}
  \end{subfigure}
  \caption{Box-plot of computation times in seconds for all the 50 instances for $\varepsilon \in \{0.5, 0.7, 0.9, 0.99\}$.}
  \label{fig:time}
\end{figure*}

\section{Conclusion} \label{sec:conclusion}
The article presents a countably infinite family of convex models--parameterized by $\varepsilon\in[0,1]$--to control the dispersion of a subset of decision variables in an optimization problem.  The parameter $\varepsilon$ bounds dispersion: when $\varepsilon=0$, dispersion is unconstrained; when $\varepsilon=1$, dispersion is forced to zero.  This family is indexed by integers $p \geqslant 2$, and all the models in the family are derived from first principles using the fundamental result of equivalence of finite-dimensional norms. For $p = 2$, the model is a second‑order cone; for $2<p<\infty$, it is an $\ell_p$‑norm cone; and for $p = \infty$, it reduces to a set of linear inequalities. The article also establishes several theoretical properties that align with intuitive notions of fair solutions. The $p = \infty$ variant is appealing computationally since the model is linear.  
We then evaluate the trade‑offs of these models on the MAAP.  In particular, the $p = \infty$ variant is appealing both computationally---since the model is linear---and in terms of dispersion control, as increasing $\varepsilon$ yields a more gradual reduction in dispersion.  
Finally, these models can be integrated into optimization models in many application domains to control dispersion or enforce the fairness of a vector of decision variables with minimal modifications to existing solution algorithms.

\begin{appendices}

\section{Proof of Proposition \ref{prop:decreasing}} \label{app:proof}
We start with introducing additional notation and a few lemmas that will make the presentation of the proof of the main result cleaner. For any vector $\bm x \in \Delta_n \setminus \left\{ \bm e_1, \dots, \bm e_n, \tfrac{1}{n}\bm e\right\}$, we let 
\begin{gather}
S_p = \sum_{i=1}^n x_i^p ~ \Rightarrow ~~ S_p^{\frac 1p} = \|\bm x\|_p \label{eq:Sr} 
\end{gather}
and
\begin{gather} 
w_i = \frac{x_i^p}{S_p} ~~ \Rightarrow ~~ \sum_{i=1}^n w_i = 1 \text{ and } \sum_{i=1}^n w_i^{\frac 1p} = \frac 1{\|\bm x\|_p}. \label{eq:w}
\end{gather}
Also, applying $\mathcal {NE}(1, p)$ to $\bm x $ yields 
\begin{flalign}
    &\|\bm x\|_p < 1 < (D_p+1)~\|\bm x\|_p 
    ~~\Longrightarrow~~ \|\bm x\|_p \in \left((D_p + 1)^{-1}, 1\right) \label{eq:xr_bounds}
\end{flalign}
Observe that $\bm x = \bm e_i$ for any $i \in \{1, \dots, n\}$ if and only if $\|\bm x\|_p = 1$ and $\bm x = \tfrac 1n \bm e$ if and only if $\|\bm x\|_p = (D_p + 1)^{-1}$. Since these values are excluded from the possible values $\bm x$ can take, the inequality in \eqref{eq:xr_bounds} is strict.  
\begin{lemma} \label{lem:log_bound}
    For any $\bm x \in \Delta_n \setminus \left\{ \bm e_1, \dots, \bm e_n, \tfrac{1}{n}\bm e\right\}$ and integer $p \geqslant 2$, the following inequality holds 
    \begin{gather}
        \frac{p}{p-1} \left[ \frac{-\ln \|\bm x\|_p}{1-\|\bm x\|_p}\right] -\ln n \left[ 1 + \frac 1{D_p}\right] < 0 \label{eq:bound-1}
    \end{gather}
\end{lemma}
\begin{proof}
    We start the proof by observing that the function 
    \begin{gather*}
        f(\|\bm x\|_p) = \frac{-\ln \|\bm x\|_p}{1-\|\bm x\|_p} ~~\text{ strictly decreases for $\|\bm x\|_p \in (0, 1)$.}
    \end{gather*}
    The above result is easy to see by evaluating the derivative and checking that the sign of the derivative is strictly negative. 
    Now, for any fixed $p \geqslant 2$, we define $t_0 = (D_p+1)^{-1}$. Since $\|\bm x\|_p\mapsto f(\|\bm x\|_p)$ is a strictly decreasing function, for every value of $\|\bm x\|_p > t_0$ we have $f(\|\bm x\|_p) < f(t_0)$, i.e., 
    \begin{flalign*} 
        -\frac{\ln \|\bm x\|_p}{1-\|\bm x\|_p} &< -\frac{\ln (D_p + 1)^{-1}}{1-(D_p+1)^{-1}} \\
        & = \frac{\ln (D_r + 1)}{1-(D_r+1)^{-1}} \\ 
        & = \frac{\ln n^{1-\frac 1r}}{1-(D_r+1)^{-1}} \\ 
        & = \frac{r-1}{r} ~ \ln n \left[ \frac{D_r + 1}{D_r}\right] \\ 
        & = \frac{p-1}{p} ~ \ln n \left[ 1 + \frac 1{D_p}\right] \\ 
    \Rightarrow ~~ \frac{p}{p-1} \left[ \frac{-\ln \|\bm x\|_p}{1-\|\bm x\|_p}\right] & < \ln n  \left[ 1 + \frac 1{D_p}\right] \\ 
    \Rightarrow ~~ \frac{p}{p-1} \left[ \frac{-\ln \|\bm x\|_p}{1-\|\bm x\|_p} \right] & -\ln n  \left[ 1 + \frac 1{D_p}\right] < 0.
    \end{flalign*}
    We remark that when $\bm x = \tfrac 1n \bm e$, $\|\bm x\|_p$ evaluates to $t_0$ and the above inequality is satisfied at equality, and the claim of the lemma does not hold. Hence, $\bm x \neq \tfrac 1n \bm e$ is assumed in the statement of the lemma.
\end{proof}
\begin{lemma} \label{lem:entropy}
    Suppose $\bm w = [w_1, \dots, w_n]^\intercal$, then for any 
    $p \geqslant 2$, the function $H(\bm w) = -\sum_i w_i \ln w_i$ satisfies 
    \begin{gather} 
        0 \leqslant H(\bm w) \leqslant -\frac{p}{p-1} \ln \|\bm x\|_p. \label{eq:bound-2}
    \end{gather}
\end{lemma}
\begin{proof}
    The definitions in \eqref{eq:w} tell us that $\bm w = [w_1, \dots, w_n]^\intercal$ is a probability vector. Hence, $H(\bm w)$ is the Shannon-entropy of $\bm w$, which we know is non-negative \cite{gallager1968information}. 
    To derive the upper bound, we use the definition of R\'enyi entropy of order $\alpha$ \cite{renyi1961measures} given below 
    \begin{gather}
        H_\alpha(\bm w) = \frac 1{1- \alpha} \ln \sum_i w_i^{\alpha} \label{eq:renyi}.
    \end{gather}
    The following two facts about $H_\alpha(\cdot)$ are well known: $H_1(\bm w) = H(\bm w)$ and
    $H_{\alpha}(\cdot)$ is a
    decreasing function of $\alpha$. 
    We now let $\alpha = \frac 1p$. Then since $\alpha < 1$ (because $p \geqslant 2$), we have 
    \begin{gather}
        H(\bm w) \leqslant H_{\frac 1p}(\bm w) =
        \frac p{p- 1} \ln \sum_{i=1}^n w_i^{\frac 1p} =
        -\frac{p}{p-1} \ln \|\bm x\|_p.
    \end{gather}
    The final equality follows from \eqref{eq:w}. 
\end{proof}

\begin{lemma} \label{lem:eq-h} 
For any integer $p \geqslant 2$ and $\bm x \in \Delta_n$, 
given the definitions in \eqref{eq:Sr}, \eqref{eq:w} and Lemma \ref{lem:entropy}, the following relationship holds 
\begin{gather*}
    \ln S_p - p \frac{S_p'}{S_p} = H(\bm w)
\end{gather*}
where $S_p'$ is the derivative of $S_p$ with respect to $p$.
\end{lemma}
\begin{proof}
    We start by computing the derivative of $S_p$ with respect to $p$ as 
    \begin{flalign*}
        S_p' & = \sum_{i=1}^n x_i^p \ln x_i = S_p \sum_{i=1}^n w_i \ln x_i \\ 
        \iff p \frac{S_p'}{S_p} & = \sum_{i=1}^n w_i \ln x_i^p = \sum_{i=1}^n w_i \ln(S_p w_i) \\ 
        \iff p \frac{S_p'}{S_p} & = \ln S_p \left( \sum_{i=1}^n w_i \right) + \sum_{i=1}^n w_i \ln w_i \\
        & = \ln S_p - H(\bm w) \\ 
        \iff H(\bm w) &= \ln S_p - r\frac{S_p'}{S_p}. 
    \end{flalign*} 
\end{proof}

\noindent We now present the proof of Proposition \ref{prop:decreasing} by performing direct sign analysis of the derivative of $F(p)$. We start by taking logarithm of $F(p)$ as 
    \begin{gather}
        \ln F(p) \triangleq g(p) = \ln(1-\|\bm x\|_p) - \ln \|\bm x\|_p - \ln D_p. \label{eq:log}
    \end{gather}
    We now let $F'(p)$, $g'(p)$, $\|\bm x\|_p'$ and $D'_p$ denote the derivative with respect to $p$ of $F(p)$, $g(p)$, $\|\bm x\|_p$ and $D_p$, respectively. Then, we can differentiate \eqref{eq:log}, we get
    \begin{flalign*}
        \frac {F'(p)}{F(p)} = g'(p) = -\frac{\|\bm x\|_p'}{1-\|\bm x\|_p} - \frac{\|\bm x\|_p'}{\|\bm x\|_p} - \frac{D'_p}{D_p} 
    \end{flalign*}
    Simplifying the above equation further, we obtain
    \begin{flalign*}
        g'(p) & = -\frac{\|\bm x\|_p'}{\|\bm x\|_p(1-\|\bm x\|_p)} - \frac{D'_p}{D_p} \\ 
        & = \frac{-1}{(1-\|\bm x\|_p)} \left[\frac 1p \frac{S_p'}{S_p} - \frac 1{p^2}\ln S_p \right] - \frac{\ln n}{p^2}\left[ 1 + \frac 1{D_p}\right]\\ 
        & = \frac{1}{r^2} \left[ \frac{\ln S_r - r \tfrac{S_r'}{S_r}}{1-\|\bm x\|_r}\right]  - \frac{\ln n}{r^2}\left[ 1 + \frac 1{D_r}\right] \\
        & = \frac{H(\bm w)}{p^2(1-\|\bm x\|_p)} - \frac{\ln n}{p^2}\left[ 1 + \frac 1{D_p}\right] ~~ \text{(using Lemma \ref{lem:eq-h})}\\ 
        & \leqslant \frac 1{p^2} \left( \frac p{p-1} \left[\frac{-\ln \|\bm x\|_p}{1-\|\bm x\|_p}\right] - \ln n\left[ 1 + \frac 1{D_p}\right]\right) \\ 
        & < 0 \quad \text{(using Lemma \ref{lem:log_bound})}.
    \end{flalign*}
    The penultimate inequality follows from Lemma \ref{lem:entropy}. We have proved that $g'(p) < 0$. But $F'(p) = F(p) \cdot g'(p)$ and since $F(p) > 0$ for any $\bm x$ that satisfies the assumptions of the proposition, we obtain $F'(p) < 0$. \hfill \qed

\end{appendices}

\begin{acknowledgements}
The authors thank  Dirk Lauinger of the MIT Energy Initiative for insightful discussions on the potential for expanding on a family of fairness models.
\end{acknowledgements}

%
\section*{Conflict of interest}
The authors declare that they have no conflict of interest.

\bibliographystyle{spmpsci}      
\bibliography{refs}   

\end{document}